\documentclass[12pt]{amsart}
\usepackage{amsmath}
\usepackage{graphicx}
\usepackage[colorlinks=true]{hyperref}
\usepackage[all]{xy}
\usepackage{amssymb}
\hypersetup{urlcolor=blue, citecolor=red}

\newtheorem{theorem}{Theorem}[section]
\newtheorem{corollary}[theorem]{Corollary}
\newtheorem{lemma}[theorem]{Lemma}
\theoremstyle{definition}
\newtheorem{definition}[theorem]{Definition}
\newtheorem*{remark}{Remark}

\title[Kolmogorov-Sinai entropy]
      {Kolmogorov-Sinai entropy via separation properties of order-generated $\sigma$-algebras}


\thanks{The authors were supported by Marie Curie Actions - International Research Staff Exchange Scheme (IRSES) FP7-People-2011-IRSES. Project number 295164.}
\thanks{The first author was also partially supported by the grant no. 01-01-12 of National Academy of Sciences of Ukraine (under the joint Ukrainian-Russian project of NAS of Ukraine and Russian Foundation of Basic Research)}

\newcommand\AAA{\mathcal{A}}
\newcommand\BBB{\mathcal{B}}
\newcommand\CCC{\mathbb{C}}
\newcommand\RRR{\mathbb{R}}
\newcommand\NNN{\mathbb{N}}

\newcommand\C{M}
\newcommand\tC{\widetilde{\C}}
\newcommand\Wman{B} 
\newcommand\Xman{\Omega}
\newcommand\Yman{X}
\newcommand\yy{x}

\newcommand\Tsubs{A}

\newcommand\Borel{\mathcal{B}}
\newcommand\FOmega{\mathcal{F}}
\newcommand\Tmap{T}
\newcommand\Xfunc{\xi}
\newcommand\Xmap{\Theta}

\newcommand\ms{\mu}
\newcommand\Partit{\mathcal{P}}

\newcommand\sgmalg[2]{\Sigma^{#1,\,#2}}
\newcommand\SAXfuncT{\sgmalg{\Xfunc}{\Tmap}} 
\newcommand\SAXfunciT{\sgmalg{\Xfunc_i}{\Tmap}} 
\newcommand\SAXmapT{\sgmalg{\Xmap}{\Tmap}} 

\newcommand\Opart[1]{\mathcal{O}_{#1}}
\newcommand\Opi[1]{O_{#1}}
\newcommand\Lm[1]{\Lambda_{#1}}

\newcommand\Ppi[3]{P^{#1,\,#2}_{#3}}
\newcommand\PpiXT[1]{\Ppi{\Xfunc}{\Tmap}{#1}}

\newcommand\Pd[3]{\Partit^{#1,\,#2}_{#3}}
\newcommand\PdXT[1]{\Pd{\Xfunc}{\Tmap}{#1}}

\newcommand\Id[1]{I_{#1}}
\newcommand\sublev[1]{K_{#1}}

\newcommand\subcirc{\stackrel{\circ}{\subset}}
\newcommand\supcirc{\stackrel{\circ}{\supset}}
\newcommand\equcirc{\stackrel{\circ}{=}}

\newcommand\subSX{\mathcal{A}}

\newcommand\qq{a}
\newcommand\qqm{\qq_{*}} 
\newcommand\qqp{\qq^{*}} 
\newcommand\bb{b}
\newcommand\cc{c}

\newcommand\noninj[1]{N_{#1}}

\newcommand\Hmu{H_{\mu}}
\newcommand\hKS{h^{KS}}
\newcommand\tk[2]{\tau_{#1}(#2)}

\newcommand\cmgset{\mathcal{V}}

\newcommand\intr[1]{C_{#1}}

\newcommand\Mnk{M(n,k)}
\newcommand\MA{A}

\begin{document}
\maketitle

\centerline{\scshape Alexandra Antoniouk}
\medskip
{\footnotesize
 \centerline{Institute of Mathematics of NAS of Ukraine}
 \centerline{Tereshchenkivs'ka str., 3, 01601 Kyiv, Ukraine}
 \centerline{email: \texttt{antoniouk.a@gmail.com}}
} 

\medskip

\centerline{\scshape Karsten Keller}
\medskip
{\footnotesize
 \centerline{Universit\"at zu L\"ubeck, Institut f\"ur Mathematik}
 \centerline{Ratzeburger Allee 160, 23562 L\"ubeck, Germany}
 \centerline{email: \texttt{keller@math.uni-luebeck.de}}
}
\medskip

\centerline{\scshape Sergiy Maksymenko}
\medskip
{\footnotesize
 \centerline{Institute of Mathematics of NAS of Ukraine}
 \centerline{Tereshchenkivs'ka str., 3, 01601 Kyiv, Ukraine}
 \centerline{email: \texttt{maks@imath.kiev.ua}}
}

\bigskip


\begin{abstract}
In a recent paper, K.~Keller has given a characterization of the Kolmo\-gorov-Sinai entropy of a discrete-time measure-preserving dynamical system on the base of an increasing sequence of special partitions. These partitions are constructed from order relations obtained via a given real-valued random vector, which can be interpreted as a collection of observables on the system and is assumed to separate points of it. In the present paper we relax the separation condition in order to generalize the given characterization of Kolmogorov-Sinai entropy, providing a statement on equivalence of $\sigma$-algebras. On its base we show that in the case that a dynamical system is living on an $m$-dimensional smooth manifold and the underlying measure is Lebesgue absolute continuous, the set of smooth random vectors of dimension $n>m$ with given characterization of Kolmo\-gorov-Sinai entropy is large in a certain sense.
\end{abstract}

\section{Introduction}

\subsection{Motivation}

Kolmogorov-Sinai entropy of a $\ms$-preserving map $\Tmap$ on a probability space $(\Xman,\FOmega,\ms)$ is an important concept in dynamical systems and ergodic theory.
It is defined as the supremum of the entropy rates $h_{\mu}(T,\subSX)$ of \emph{all} finite partitions $\subSX\subset \FOmega$ of $\Xman$, which usually makes its determination complicated.
In some exceptional cases, a generating partition is known allowing to determine the Kolmogorov-Sinai entropy on the base of only this partition (see, e.g., \cite{Walters:GTM:1982}), but generally one has to take into account an infinite collection of finite partitions.

Here the question arises whether such a collection is given in a natural way. An interesting approach leading to some kind of natural partitioning
was given by introducing the concept of permutation entropy by C.~Bandt and B.~Pompe \cite{BandtPompe2002} (see also \cite{Amigo:2010}).
This quantity is based on only considering the order structure of a system and has been applied to the analysis of long time series, for example, of electroencephalograms and cardiograms.
The point that Kolmogorov-Sinai entropy and permutation entropy coincide for piecewise monotone interval maps, as shown by C.~Bandt, G.~Keller and B.~Pompe \cite{BandtKellerPompe:Nonlin:2002}, gives rise to the question if both entropies are equivalent for a broader class of dynamical systems.

\begin{remark}
J.~Amig\'{o}, M.~Kennel, and L.~Kocarev \cite{AmigoKennelKocarev2005, Amigo2012} have shown equivalence of Kolmogorov-Sinai entropy to a modified concept of permutation entropy which is structurally similar to that of Kolmogorov-Sinai entropy.
\end{remark}

K.~Keller and M.~Sinn \cite{KellerSinn:Nonlin:2009, KellerSinn2010, Keller:DCDS:2012}
have discussed the question of coincidence of permutation entropy and Kolmogorov-Sinai entropy in a general context, in particular by considering dynamical systems equipped with a random vector
\begin{eqnarray*}
\Xmap=(\Xfunc_1,\Xfunc_2,\ldots,\Xfunc_n):\Xman\to{\mathbb R}^n.
\end{eqnarray*}
Here the idea is to measure complexity of a system via the `observables' $\Xfunc_1,\Xfunc_2,\ldots,\Xfunc_n$. For given $d\in {\mathbb N}$, the set $\Xman$ is partitioned into sets of points $\omega\in\Xman$ for which
all vectors \[(\Xfunc_i(\omega),\Xfunc_i(T(\omega)),\ldots, \Xfunc_i(T^d(\omega)), \quad i=1,2,\ldots ,d\] are of the same order type.
The larger $d$ the more information on the system is given by the partition obtained in this way and called $\Pd{\Xmap}{\Tmap}{d}$ here. The permutation entropy is defined as the upper limit of the Shannon entropy of the $\Pd{\Xmap}{\Tmap}{d}$ relative to $d$ for $d\rightarrow\infty$.

It has been shown that under certain `separation' conditions on $(\Tmap, \Xmap)$ it holds
\begin{eqnarray}\label{mainform}
\hKS_{\mu}(\Tmap) = \lim\limits_{d\to\infty} h_{\mu}(T,\Pd{\Xmap}{\Tmap}{d}),
\end{eqnarray}
and that the permutation entropy with respect to $\Xfunc$ is not less than the Kolmogorov-Sinai entropy. Under validity of \eqref{mainform}, the problem of equality of both entropies is reduced to a combinatorial problem related to the problem of equality of permutation entropy and the right side of \eqref{mainform} (see K.~Keller, A.~Unakafov and V.~Unakafova \cite{KellerUnakafovUnakafova2012}). Therefore, it is of some particular interest to find sufficient conditions for \eqref{mainform} being as general as possible. This is the central aim of the present paper.

\subsection{An outline}

The main ingredient for showing \eqref{mainform} is the equivalence of two $\sigma$-algebras with respect to $\mu$ in the case of ergodic $T$:
\begin{eqnarray}\label{mainform2}
\SAXmapT \equcirc \FOmega,
\end{eqnarray}
where $\SAXmapT$ is the $\sigma$-algebra generated by $\bigcup_{d=1}^{\infty}\, \Pd{\Xmap}{\Tmap}{d}$.
For making  apparent the structural arguments, consider the third $\sigma$-algebra $\sigma \bigl(\{\Xmap\circ\Tmap^k\}_{k\geq0}\bigr)$, generated by
$\Xmap$ and their `shifts' $\Xmap\circ T,\Xmap\circ T^2,\ldots$. The central statement of this paper is that for ergodic $T$
\begin{eqnarray}
\sigma \bigl(\{\Xmap\circ\Tmap^k\}_{k\geq0}\bigr) \subcirc\SAXmapT.
\end{eqnarray}
Since
\begin{eqnarray}
\SAXmapT\subset \sigma \bigl(\{\Xmap\circ\Tmap^k\}_{k\geq0}\bigr),
\end{eqnarray}
which can be verified by standard arguments, this provides that in the ergodic case
\begin{eqnarray}\label{weaker}
\sigma \bigl(\{\Xmap\circ\Tmap^k\}_{k\geq0}\bigr) \equcirc \FOmega
\end{eqnarray}
is equivalent to \eqref{mainform2}, hence sufficient for \eqref{mainform}.
The second ingredient for showing \eqref{mainform} is ergodic decomposition.

Condition \eqref{weaker} is substantially weaker than the corresponding statement in \cite{Keller:DCDS:2012}, allowing generalizations of consequences of the main statement therein. In particular, the application of embedding theory (compare \cite{Takens81} and \cite{Sauer91}) is more apparent from the viewpoint of our paper, but it also turns out that the full power of this theory is not needed. In this paper we will show that the set of smooth maps $\Xmap$ satisfying \eqref{mainform}
which are not too far from being injective is large in a certain sense.

\subsection{Organization of the paper}

In Section \ref{s2} we give the basic definitions and formulate the main statements of the paper, which are Theorems
\ref{th:sigmaX_F_hKS_limPE} and \ref{th:setX_residual}.

Section \ref{s3} is mainly devoted to the proof of Theorem
\ref{th:sigmaX_F_hKS_limPE}. Here, the ideas given in \cite{Keller:DCDS:2012} are lifted to a sufficiently abstract level, in order to extract the general structures and to find the necessary assumptions under which \eqref{mainform} is satisfied.

The proof of Theorem \ref{th:setX_residual} is given in Section \ref{s4}. As a preparation of the proof, we recall some definitions and statements from (differential) topology, as for example the Multijets transversality theorem, and deduce some statement being interesting from their own right.

%

\section{Preliminaries and formulation of main results}\label{s2}

\subsection{Kolmogorov-Sinai entropy}

Let $\Xman$ be a non-empty set.
For a family of subsets $\AAA=\{A_i\}_{i\in I}$ of $\Xman$, denote by $\sigma(\AAA)$ the $\sigma$-algebra generated by $\AAA$.

If $\Xmap:\Xman\to\Yman$ is a map into some topological space $\Yman$, then we denote by $\sigma(\Xmap)$ the $\sigma$-algebra on $\Xman$ of inverse images of the $\sigma$-algebra $\Borel(X)$ of Borel subsets of $\Yman$ under $\Xmap$.

If $\AAA=\{A_i\}_{i\in I}$ and $\BBB=\{B_j\}_{j\in J}$ are two partitions of $\Xman$, then we define the new partition $\AAA \vee \BBB$ of $\Xman$ by
\[\AAA \vee \BBB = \{A_i \cap B_j \mid A_i\in \AAA, \ B_j\in \BBB\}.\]
We write \[\AAA\prec\BBB\] if each element $A\in\AAA$ is a finite union of some elements of $\BBB$.

Let $\FOmega$ be a $\sigma$-algebra of subsets of $\Xman$ and $\mu$ be a measure on $\FOmega$.
Denote by $\Pi(\FOmega)$ the set of all \emph{finite partitions} $\AAA=\{A_1,\ldots,A_n\}$ of $\Xman$ such that $A_i\in\FOmega$ for each $i=1,\ldots.n$.
Then the \emph{entropy} of $\AAA \in \Pi(\FOmega)$ with respect to $\mu$ is defined by the formula
\[
 \Hmu(\AAA) = -\sum_{i=1}^{n} \mu(A_i) \log\mu(A_i).
\]
Further, let $\Tmap:\Xman\to\Xman$ be a measurable map.
Denote by $\Tmap^{-1}\AAA$ the partition of $\Xman$ consisting of all inverse images of elements of $\AAA$:
\[
 \Tmap^{-1}\AAA = \{ \Tmap^{-1}(A_1), \ldots, \Tmap^{-1}(A_n)\}.
\]
For each $k\geq1$ define the partition
\[
 \tk{k}{\AAA} = \AAA\,\vee\, \Tmap^{-1}\AAA \,\vee\, \cdots\, \Tmap^{-(k-1)}\AAA.
\]
Evidently,
\[
\tk{1}{ \tk{k}{\AAA} } = \tk{k+1}{\AAA}.
\]

\begin{definition}
Let $\Tmap:\Xman\to\Xman$ be a measurable map.
Then its \emph{Kolmogorov-Sinai} entropy is defined by the formula:
\[
 \hKS_{\mu}(\Tmap) = \sup_{\AAA\in \Pi(\FOmega)} \lim_{k\to\infty} \frac{1}{k}\Hmu(\tk{k}{\AAA}).
\]
\end{definition}

Though the computation of Kolmogorov-Sinai entropy requires considering \emph{all} finite partition of $\Xman$ belonging to $\Pi(\FOmega)$, the following lemma shows that this entropy can be obtained from certain increasing sequences of finite partitions.
\begin{lemma}\label{lm:hKS_limd}{\rm\cite[Lemma~4.2]{Walters:GTM:1982}}
Let $\{\AAA_d\}_{d\geq1}$ be a sequence of finite partitions of $\FOmega$ such that
\[
 \AAA_{1} \prec \AAA_{2} \prec \cdots \prec \AAA_{d} \prec \cdots
\]
and $\sigma\left(\{\AAA_d\}_{d=1}^{\infty}\right) \equcirc \FOmega$.
Then $\hKS_{\mu}(\Tmap) = \lim\limits_{d\to\infty} \lim\limits_{k\to\infty} \frac{1}{k}\Hmu(\tk{k}{\AAA_d})$.
\end{lemma}

If $\AAA,\BBB\subset\FOmega$ are two sub-$\sigma$-algebras, we
write $\BBB\subcirc\AAA$ if for each $B\in\BBB$ there exists some $A\in\AAA$ such that $\mu(B\bigtriangleup A)=0$.
Correspondingly, we write $\BBB\equcirc\AAA$ if $\AAA\subcirc\BBB$ and $\BBB\subcirc\AAA$.

\subsection{Ordinal partition $\Opart{d}$ of $\RRR^{d+1}$}

For a permutation $\pi=(i_0,\ldots,i_{d})$ of a set $\{0,\ldots,d\}$ define the subset $\Opi{\pi}$ of $\RRR^d$ by the following rule: the point $(x_0,\ldots,x_{d})\in\RRR^{d+1}$ belongs to $\Opi{i_0,\ldots,i_{d}}$ whenever
\[
 x_{i_0} \geq x_{i_1} \geq \cdots \geq x_{i_{d}}
\]
and if $x_{i_{\tau}} = x_{i_{\tau+1}}$ for some $\tau\in\{0,\ldots,d-1\}$ then
\[
 i_{\tau} > i_{\tau+1}.
\]
\begin{remark}
Notice that each vector $x = (x_0,\ldots,x_{d})\in\RRR^{d+1}$ can be regarded as a
$(d+1)$-tuple of pairs of numbers:
\begin{equation}\label{equ:set_of_pairs}
 \bigl(\, (x_0,0),\, (x_1, 1),\, \ldots,\, (x_{d},d)\, \bigr).
\end{equation}
This set can be \emph{uniquely} lexicographically ordered in a decreasing manner: at first we sort them by values of $x_i$, and then by their indices $i$.
Thus we can associate to $x$ a \emph{unique} permutation $\pi$ of indexes $\{0,\ldots,d\}$ which sorts the above set of pairs~\eqref{equ:set_of_pairs}.
Then $\Opi{\pi}$ consists of all $x\in\RRR^{d+1}$ that can be sorted by the same permutation $\pi$.
\end{remark}

It is easy to see that the following family of sets
\[
  \Opart{d} = \left\{ \Opi{\pi} \mid \text{$\pi=(i_0,\ldots,i_{d})$ is a
  permutation of $\{0,\ldots,d\}$}\right\}
\]
is a partition of $\RRR^{d+1}$.

\subsection{Ordinal partition of $\Xman$}

Now let $\Xman$ be a set, $\Tmap:\Xman\to\Xman$ and $\Xfunc:\Xman\to\RRR$ be a function.
Then for each $d\in\NNN$ we can define the following map
\[
 \Lm{d} = (\Xfunc, \Xfunc\circ\Tmap, \ldots, \Xfunc\circ\Tmap^{d}):
 \Xman \to \RRR^{d+1}.
\]
Define the partition $\PdXT{d} = \{ \PpiXT{\pi} \}$ of $\Xman$, where
\[
 \PpiXT{\pi} = \Lm{d}^{-1}(\Opi{\pi}),
\]
and $\pi$ runs over all permutations of the set $\{0,\ldots,d\}$.
Thus $\PdXT{d}$ is just the inverse image of the partition $\Opart{d}$ of $\RRR^{d+1}$ under the map $\Lm{d}$.

\begin{remark}\rm
Notice that each set $\PpiXT{\pi}$, $\pi=(i_0,\ldots,i_{d})$ consists of all $\omega\in\Xman$ such that
\[
\Xfunc\circ\Tmap^{i_0}(\omega) \ \geq \ \Xfunc\circ\Tmap^{i_1}(\omega) \ \geq \ \cdots \ \geq \ \Xfunc\circ\Tmap^{i_{d}}(\omega),
\]
and if $\Xfunc\circ\Tmap^{i_{\tau}}(\omega)=\Xfunc\circ\Tmap^{i_{\tau+1}}(\omega)$, then $i_{\tau}>i_{\tau+1}$.
\end{remark}
\begin{remark}\rm
The partition $\PdXT{d}$ can also be described in the following way.
For each pair $(i,j)$ such that $0\leq i < j \leq d$ define the partition of $\Omega$ by two sets:
\begin{equation}\label{equ:Rp_Rm}
\begin{aligned}
 R_{d}^{i,j} &= \{ \omega\in\Xman \ | \ \Xfunc\circ\Tmap^{i}(\omega) < \Xfunc\circ\Tmap^j(\omega) \}, \\
 R_{d}^{j,i} &= \{ \omega\in\Xman \ | \ \Xfunc\circ\Tmap^{i}(\omega) \geq \Xfunc\circ\Tmap^j(\omega) \}.
\end{aligned}
 \end{equation}
Then it is easy to see that
\begin{equation}\label{equ:Pd_via_Rd}
 \PdXT{d} \, = \, \bigvee_{i\not=j \in \{0,\ldots, d\}} R_{d}^{i,j}.
\end{equation}
\end{remark}


\begin{definition}
The $\sigma$-algebra
\[
  \SAXfuncT \ = \ \sigma\bigl( \bigl\{\PdXT{d}\bigr\}_{d=1}^{\infty} \bigr)
\]
is called the \emph{ordinal} $\sigma$-algebra of $\Xman$ for $(\Xfunc,\Tmap)$.
\end{definition}

More generally, let $\Xmap=(\Xfunc_1,\ldots,\Xfunc_n):\Xman\to\RRR^n$ be a map.
Then we define the partition
\[
 \Ppi{\Xmap}{\Tmap}{d} = \mathop\bigvee\limits_{i=1}^{n}\Ppi{\Xfunc_i}{\Tmap}{d}, \qquad d\geq 1,
\]
and the $\sigma$-algebra
\[
 \SAXmapT := \sigma\left( \bigl\{ \Ppi{\Xmap}{\Tmap}{d} \bigr\}_{d=1}^{\infty} \right)
 = \sigma\left( \bigl\{ \sgmalg{\Xfunc_i}{\Tmap} \bigr\}_{i=1}^{n} \right),
\]
which we call the \emph{ordinal} $\sigma$-algebra of $\Xman$ for $(\Xmap,\Tmap)$.

Suppose $\FOmega$ is a $\sigma$-algebra of subsets of $\Xman$ such that $\Tmap:\Xman\to\Xman$ is $\FOmega$-$\FOmega$-measurable and $\Xmap:\Xman\to\RRR^n$ is $\FOmega$-$\Borel(\RRR^n)$-measurable.
Then it is obvious that
\[
\SAXmapT \subset \FOmega.
\]
The following lemma easily follows from~\eqref{equ:Rp_Rm} and~\eqref{equ:Pd_via_Rd} and we left it to the reader.
\begin{lemma}\label{lm:SigmaXT_sigmaXT}
$\SAXmapT\subset\sigma\bigl(\{\Xmap\circ\Tmap^k\}_{k\geq0}\bigr)$.
\qed
\end{lemma}

\subsection{Main results}

The following theorem gives sufficient conditions for the validity of~\eqref{mainform}.
\begin{theorem}\label{th:sigmaX_F_hKS_limPE}
Let $(\Xman, \FOmega, \mu)$ be a probability space, $\Tmap:\Xman\to\Xman$ be a measurable $\mu$-invariant transformation, and $\Xmap=(\Xfunc_1,\ldots,\Xfunc_n):\Xman\to\RRR^n$ be a measurable map such that $\sigma \bigl(\{\Xmap\circ\Tmap^k\}_{k\geq0}\bigr) \equcirc \FOmega$.
Suppose also that one of the following conditions holds true: either
\begin{enumerate}
 \item[\rm(a)]
 $\Tmap$ is ergodic, or
 \item[\rm(b)]
 $\Tmap$ is not ergodic, however $\Xman$ can be embedded into some compact metrizable space so that $\FOmega=\Borel(\Xman)$.
\end{enumerate}
Then
\[
 \hKS_{\mu}(\Tmap) =  \lim_{d\to\infty} \lim_{k\to\infty}\,\frac{1}{k}\, \Hmu\left(\tk{k}{\Ppi{\Xmap}{\Tmap}{d}}\right).
\]
\end{theorem}

We will now recall the notions of \emph{residuality} and \emph{prevalence} being respectively a topological and a measure-theoretic formalization of the expression ``almost every''.

A subset $A$ of a topological space is \emph{residual} if $A$ is an intersection of countably many sets with dense interiors.
A \emph{Baire} space is a topological space in which every residual subset is dense.
\emph{Every complete metric space is Baire}.

\begin{definition}
Let $V$ be a linear vector space over $\RRR$ or $\CCC$.
A finite-dimensional subspace $P\subset V$ is called a \emph{probe} for a set $\Tsubs\subset V$ if for each $v\in V$ the intersection $P\cap [ (V\setminus \Tsubs)+v]$ has Lebesgue measure zero in $P$.
\end{definition}

Suppose now that $V$ is a \emph{topological vector space}, i.e.~that it has a topology in which addition of vectors and multiplication by scalars are continuous operations.

\begin{definition}
Let $\mu$ be a nonnegative measure on the $\sigma$-algebra $\Borel(V)$, and $S\subset V$ be a Borel subset.
Then $\mu$ is said to be \emph{transverse} to $S$ if the following two conditions hold:
\begin{enumerate}
\item[\rm(i)]
There exists a compact subset $U\subset V$ such that $0 < \mu(U) < \infty$.
\item[\rm(ii)]
$\mu(S+v)=0$ for every $v\in V$, where $S+v=\{s+v \mid s\in S\}$ is the translation of $S$ by the vector $v$.
\end{enumerate}
\end{definition}
\begin{definition}


A subset $\Tsubs\subset V$ is called \emph{prevalent} if its complement $V\setminus\Tsubs$ is contained in some Borel set admitting a transverse measure.

\end{definition}

The following lemma summarizes some properties of prevalent sets obtained in~\cite{HuntSauerYourke:BAMS:1992}.
\begin{lemma}\label{lm:prevalency_prop}\cite{HuntSauerYourke:BAMS:1992}
Suppose $V$ admits a complete metric.
Let also $\Tsubs \subset V$ be a subset.
\begin{enumerate}
\item[\rm 1)]
If $\Tsubs$ is prevalent, then $V\setminus \Tsubs$ is nowhere dense.
\item[\rm 2)]
If $\dim V < \infty$, then $\Tsubs$ is prevalent if and only if $V\setminus \Tsubs$ has Lebesgue measure zero.
\item[\rm 3)]
If $\dim V =\infty$ and $V \setminus \Tsubs$ is compact, then $\Tsubs$ is prevalent.
\item[\rm 4)]
If $\Tsubs$ admits a probe, then $\Tsubs$ is prevalent. \qed
\end{enumerate}
\end{lemma}

In general classes of residual and prevalent subsets of a complete metric space $V$ are distinct and no one of them contains the other.

Let $\Xman$ and $\Yman$ be smooth manifolds and $r=0,1,\ldots,\infty$.
Then the space $C^{r}(\Xman,\Yman)$ admits two natural topologies \emph{weak}, $C^{k}_W$, and \emph{strong}, $C^{k}_S$.
The following lemma collects some information about these topologies, see e.g.~\cite[Chapter~2]{Hirsch:DiffTop:1976} and~\cite[Chapter II, \S3]{GolubitskyGuillemin:GTM:1973}.
\begin{lemma}\label{lm:topologies}{\rm }
{\rm1)} Topology $C^{r}_S$ is finer than $C^{r}_W$.
If $\Xman$ is compact, then these topologies coincide.
%

{\rm2)} $C^{r}(\Xman,\Yman)$ is a Baire space with respect to each of the topologies $C^{r}_{W}$ and $C^{r}_{S}$.

{\rm3)} $C^{r}(\Xman,\Yman)$ admits a complete metric with respect to the weak topology $C^{r}_W$.

{\rm 4)}
Suppose $\Yman=\RRR^n$, so the space $C^{r}(\Xman,\RRR^n)$ has a natural structure of a linear space.
Then $C^{r}(\Xman,\RRR^n)$ is a topological vector space with respect to the weak topology $C^{\infty}_{W}$.
However, if $\Xman$ is non-compact, then $C^{r}(\Xman,\RRR^n)$ is {\bfseries not a topological vector space} with respect to the strong topology $C^{r}_{S}$, since the multiplication by scalars is not continuous.
\qed
\end{lemma}

Again let $\Xman$ be a smooth manifold of dimension $m$.
\begin{definition}\label{defn:zero_measure}\rm
A subset $D\subset\Xman$ \emph{has measure zero} if for any local chart $(U,\varphi)$ on $\Xman$, where $U\subset \Xman$ is an open subset and $\varphi: U\to\RRR^m$ is a smooth embedding, the set $\varphi(D \cap U)$ has Lebesgue measure zero in $\RRR^m$.
\end{definition}

\begin{definition}\label{defn:Leb_abs_cont}\rm
Let $\mu$ be a measure on $\Borel(\Xman)$.
We will say that $\mu$ is \emph{Lebesgue absolute continuous} if $\mu(D)=0$ for any subset $D\subset\Xman$ of measure zero in the sense of Definition~\ref{defn:zero_measure}.
\end{definition}

\begin{remark}\rm
We can reformulate Definition~\ref{defn:Leb_abs_cont} as follows.
Let $\lambda$ be a Lebesgue measure on $\RRR^m$ and $(U,\varphi)$ be a local chart on $\Xman$.
Since $\varphi$ is an embedding, we can define the induced measure $\varphi_{*}(\lambda)$ on $\Borel(U)$ by $\varphi_{*}(\lambda)(A) = \lambda(A)$ for all $A\in\Borel(U)$.
Then $\mu$ is \emph{Lebesgue absolute continuous} if for any local chart $(U,\varphi)$ the restriction of $\mu$ to $\Borel(U)$ is absolute continuous with respect to $\varphi_{*}(\lambda)$.
\end{remark}

Our second result shows that the set of maps $\Xmap$ for which~\eqref{mainform} holds is ``large''.
\begin{theorem}\label{th:setX_residual}
Let $\Xman$ be a smooth manifold of dimension $m$, $\mu$ be a measure on $\Borel(\Xman)$, $\Tmap:\Xman\to\Xman$ be a measurable $\mu$-invariant transformation.
Suppose $\mu$ is Lebesgue absolute continuous in the sense of Definition~\ref{defn:Leb_abs_cont}.
Let $\cmgset$ be the set of all maps $\Xmap\in C^{\infty}(\Xman,\RRR^n)$ for which
\begin{equation}\label{hKS_is_a_limit}
 \hKS_{\mu}(\Tmap) =  \lim_{d\to\infty} \lim_{k\to\infty}\,\frac{1}{k}\, \Hmu\left(\tk{k}{\Ppi{\Xmap}{\Tmap}{d}}\right)
\end{equation}
holds.
If $n>m$, then $\cmgset$ is {\bfseries residual} in $C^{\infty}(\Xman,\RRR^n)$ with respect to strong topology $C^{\infty}_{S}$, and {\bfseries prevalent} with respect to the weak topology $C^{\infty}_W$.
\end{theorem}



\section{Separation via $\sigma$-algebras}\label{s3}

\subsection{Properties of distribution functions}


Let $(\Xman, \FOmega, \mu)$ be a probability space and $\Xfunc:\Xman\to\RRR$ be a measurable function.
Let also $F:\RRR\to[0,1]$ be the distribution function of $\Xfunc$, i.e.
\begin{align*}
 F(\qq) &= \mu\{\omega \mid \Xfunc(\omega)\leq \qq  \} = \mu \bigl( \Xfunc^{-1}(-\infty,\qq] \bigr)
\end{align*}
It is well-known that $F$ is non-decreasing, right continuous, and that
\[
 \lim_{\qq\to-\infty} F(\qq) = 0, \qquad  \lim_{\qq\to+\infty} F(\qq) = 1.
\]
The latter justifies that $F$ can also be considered as a function from
$[-\infty,+\infty]$ into $[0,1]$.

For further considerations it will be convenient to keep in mind the following commutative diagram:
\begin{equation}\label{equ:X_F_FX}
 \xymatrix{
 \Xman \ar[rr]^{\Xfunc} \ar[dr]_{F\circ\Xfunc} && \RRR \ar[dl]^{F} \\
 & [0,1]}
\end{equation}
which implies that $F\circ\Xfunc$ is $\FOmega$-$\Borel(\RRR)$-measurable and so it holds $\sigma(F\circ\Xfunc)\subset\FOmega$.

For each $\qq\in\RRR$ define the following two elements of $[-\infty,\infty]$:
\[
 \qqm = \inf(F^{-1}F(\qq)),
 \qquad
 \qqp = \sup(F^{-1}F(\qq)).
\]

Moreover, let
\[
 \intr{\qq} = (-\infty,\qq).
\]

\begin{lemma}\label{lm:distr_func_prop}
Let $\qq\in\RRR$.
Then the following statements hold true.
\begin{enumerate}
\item[\rm(1)]
$F^{-1}F(\qq)$ coincides either with $[\qqm,\qqp)$ or with $[\qqm,\qqp]$.
\item[\rm(2)]
$F^{-1} F (\intr{\qqm}) = \intr{\qqm}$.
\item[\rm(3)]
If $\qqm<\qq$, then \[F^{-1}F(\intr{\qq}) = \intr{\qqm} \cup F^{-1}F(\qq),\] so
$F^{-1}F(\intr{\qq})$ either equals $(-\infty,\qqp)$ or $(-\infty,\qqp]$.
Moreover, let
\[
 Z = F^{-1}F(\intr{\qq}) \setminus \intr{\qq}.
\]
Then $\mu(\Xfunc^{-1}(Z))=0$.
\end{enumerate}
\end{lemma}
\begin{proof}
(1)
Evidently, $F^{-1}F(\qq) \subset [\qqm,\qqp]$.

We will now prove that $[\qqm,\qqp) \subset  F^{-1}F(\qq)$.
By definition of infimum and supremum of the set $F^{-1}F(\qq)$ there are two sequences $\{x_i\}, \{y_i\} \subset F^{-1}F(\qq)$ such that
\[
\qqm \leq \cdots \leq x_{i+1} \leq x_{i} \leq \cdots \leq x_1 \leq y_1 \leq \cdots \leq y_i \leq y_{i+1} \leq \cdots \leq \qqp,
\]
$\lim\limits_{i \to\infty} x_i=\qqm$ and $\lim\limits_{i\to\infty} y_i=\qqp$.

Since $F$ is nondecreasing and $F(x_i)=F(y_i)=F(\qq)$ for all $i$, it follows that $F$ is constant on each segment $[x_i,y_i]$, and so
\[ (\qqm,\qqp) \ = \ \bigcup_{i}\,[x_i,y_i] \ \subset \ F^{-1}F(\qq).\]
Moreover, from right-continuity of $F$ we obtain that $F(\qqm) = \lim\limits_{i\to\infty} F(x_i)=F(\qq)$, hence $[\qqm,\qqp) \subseteq F^{-1}F(\qq)$.

\medskip

(2) The inclusion $\intr{\qqm} \subset F^{-1} F (\intr{\qqm})$ is evident.
Suppose that there exists some
\[t\ \in \ F^{-1} F (\intr{\qqm}) \, \setminus \, \intr{\qqm}.\]
This means that
\begin{itemize}
\item[(i)]
$t\geq\qqm$, and
 \item[(ii)]
$F(t) \in F(\intr{\qqm})$, i.e. $F(t)=F(s)$ for some $s < \qqm$,
\end{itemize}
Thus $s < \qqm \leq t$.
Since $F$ is non-decreasing,
\[ F(s) \ = \ F(\qqm) \stackrel{(1)}{=} F(a) \ = \ F(t),\]
that is $s\in F^{-1}F(\qq)$, and therefore $\qqm \leq s$, contradicting the assumption. Thus $F^{-1} F (\intr{\qqm}) = \intr{\qqm}$.

\medskip

(3)
Since $\qqm<\qq$, $F(\qqm)=F(\qq)$, and $F$ is non-decreasing, it follows that
\[
 F(\intr{\qq}) = F((-\infty,\qqm)) \cup F([\qqm,\qq)) = F(\intr{\qqm}) \cup \{ F(\qq) \},
\]
hence
\begin{align*}
 F^{-1}F(\intr{\qq}) \ &=  \ F^{-1}\bigl(F(\intr{\qqm})\, \cup \, \{F(\qq)\}\bigr)
 \\ &=F^{-1}F(\intr{\qqm}) \,\cup\, F^{-1}F(\qq)
 \ \stackrel{(2)}{=}\ \intr{\qqm} \,\cup\, F^{-1}F(\qq).
\end{align*}
It follows that $Z$ either equals $[\qq,\qqp]$ or $[\qq,\qqp)$.
Suppose $Z=[\qq,\qqp]$, then
\begin{align*}
\mu(\Xfunc^{-1}(Z))
&= \mu(\Xfunc^{-1}[\qq,\qqp]) = \mu(\Xfunc^{-1}(-\infty,\qqp]) - \mu(\Xfunc^{-1}(-\infty,\qq)) \\
&= F(\qqp) - \lim\limits_{\substack{t\to\qq \\ t\in(\qqm,\qq)}} \mu(\Xfunc^{-1}(-\infty,t]) \\
&= F(\qq) - \lim\limits_{\substack{t\to\qq \\ t\in(\qqm,\qq)}} F(t) = F(\qq) - \lim\limits_{\substack{t\to\qq \\ t\in(\qqm,\qq)}} F(\qq) = 0
\end{align*}

Now, let $Z=[\qq,\qqp)$.
Then similarly,
\begin{align*}
\mu(\Xfunc^{-1}(Z))
&= \mu(\Xfunc^{-1}[\qq,\qqp)) = \mu(\Xfunc^{-1}(-\infty,\qqp)) - \mu(\Xfunc^{-1}(-\infty,\qq)) \\
&= \lim\limits_{\substack{s\to\qqp \\ s\in(\qqm,\qqp)}} \mu(\Xfunc^{-1}(-\infty,s]) -
   \lim\limits_{\substack{t\to\qq \\ t\in(\qqm,\qq)}} \mu(\Xfunc^{-1}(-\infty,t]) \\
&= \lim\limits_{\substack{s\to\qqp \\ s\in(\qqm,\qqp)}} F(s) -
   \lim\limits_{\substack{t\to\qq \\ t\in(\qqm,\qq)}} F(t)
   = \lim\limits_{\substack{s\to\qqp \\ s\in(\qqm,\qqp)}} F(\qq) -
   \lim\limits_{\substack{t\to\qq \\ t\in(\qqm,\qq)}} F(\qq) = 0.
\end{align*}
Lemma is completed.
\end{proof}

\begin{lemma}\label{lm:sigmaX_sigma_FX}
$\sigma(F \circ \Xfunc) \equcirc \sigma(\Xfunc)$.
\end{lemma}
\begin{proof}
It is easy to see that $\sigma(F \circ \Xfunc) \subset \sigma(\Xfunc)$.
Indeed, let $A\in\sigma(F \circ \Xfunc)$, so
\[
 A = (F \circ \Xfunc)^{-1}(B) = \Xfunc^{-1} F^{-1}(B)
\]
for some $B\in\Borel([0,1])$.
But $F^{-1}(B) \in \Borel(\RRR)$, hence $A\in\sigma(\Xfunc)$.

Now we will show that $\sigma(F \circ \Xfunc) \supcirc \sigma(\Xfunc)$.
For each $\qq\in\RRR$ let
\[
 P_{\qq} = \Xfunc^{-1}(\intr{\qq}).
\]
Then $\sigma(\Xfunc)$ is generated by the sets $P_{\qq}$, so it suffices to prove that for each $\qq\in\RRR$ there exists some $Q_{\qq}\in\sigma(F \circ \Xfunc)$ such that $\mu (Q_{\qq} \bigtriangleup P_{\qq})= 0.$

In fact we will put
\[
 Q_{\qq} = \Xfunc^{-1} F^{-1}F(\intr{\qq}) = (F\circ \Xfunc)^{-1} F(\intr{\qq}).
\]
Since $F$ is non-decreasing $F(\intr{\qqm})$ is a Borel subset of $[0,1]$, hence $Q_{\qqm}\in\sigma(F \circ \Xfunc)$.
So it remains to show that $\mu (Q_{\qq} \bigtriangleup P_{\qq})= 0$ for each $\qq\in\RRR$.

First suppose $\qq=\qqm$.
Then by (2) of Lemma~\ref{lm:distr_func_prop}
\[
P_{\qq} = P_{\qqm} =\Xfunc^{-1} (\intr{\qqm})
\ \stackrel{(2)}{=\!=} \
\Xfunc^{-1} F^{-1}F(\intr{\qqm}) \ = \ Q_{\qqm} = Q_{\qq},
\]
hence $Q_{\qq} \bigtriangleup P_{\qq}=\varnothing$, and so $\mu (Q_{\qq} \bigtriangleup P_{\qq})= 0$.

Now suppose $\qqm<\qq$. Then for $Z = F^{-1}F(\intr{\qq}) \setminus \intr{\qq}$ it holds
\[
 Q_{\qq} = \Xfunc^{-1} F^{-1}F(\intr{\qq}) = \Xfunc^{-1}(\intr{\qq}) \cup  \Xfunc^{-1}(Z) = P_{\qq} \cup \Xfunc^{-1}(Z),
\]
Therefore, by (4)
\[
 \mu (Q_{\qq} \bigtriangleup P_{\qq}) = \mu (Q_{\qq} \setminus P_{\qq}) = \mu (\Xfunc^{-1}(Z)) = 0.
\]
The Lemma is proved.
\end{proof}

\subsection{Ergodic properties.}
Let $\Tmap:\Xman\to\Xman$ be a measurable map.
Define the function $\Id{d}:\Xman\to\RRR$ by
\[
 \Id{d}(\omega) =  \# \{ r=1,\ldots,d-1 \mid \Xfunc\circ\Tmap^r(\omega) \leq \Xfunc(\omega) \}.
\]
So $\Id{d}(\omega)$ is the number of points among the first $d-1$ points of the $\Tmap$-orbit of $\omega$ at which $\Xfunc$ takes values not greater than $\Xfunc(\omega)$.
\begin{lemma}\label{lm:FX_lim_Id}
If $\Tmap$ is ergodic and $\mu$-preserving, then
\begin{equation}\label{equ:FX_lim_Id}
 F\circ\Xfunc(\omega) = \mathop{\mathrm{a.e.lim}}\limits_{d\to\infty} \frac{\Id{d}(\omega)}{d}
\end{equation}
\end{lemma}
\begin{proof}
For each $\qq\in\RRR$ consider the following set
\[
 \sublev{\qq} = \Xfunc^{-1}(-\infty,\qq].
\]
Then by definition
\[
 F(\qq) = \mu (\sublev{\qq}) = \mu( \omega \in\Xman \mid \Xfunc(\omega)\leq\qq ).
\]

Moreover, as $\Tmap$ is ergodic, it follows from Birkhoff's Ergodic Theorem that there exists a subset $\Xman_{\qq} \subset\Xman$ such that $\mu(\Xman_{\qq})=1$, and for each $\bar \omega\in\Xman_{\qq}$
\begin{align*}
 \mu (\sublev{\qq})
 &= \lim_{d\to\infty} \frac{1}{d} \, \# \{r<d \mid \Tmap^r(\bar\omega) \in \sublev{\qq} \} \\
 &= \lim_{d\to\infty} \frac{1}{d} \, \# \{r<d \mid \Xfunc\circ\Tmap^r(\bar\omega) \leq \qq \}.
\end{align*}

Take any countable dense subset $S \subset\RRR$ containing all points of discontinuity of $F$ and let
\[ \bar\Xman = \mathop{\bigcap}\limits_{\qq\in S} \Xman_{\qq}. \]
Then $\mu(\bar\Xman)=1$ as well, and for each $\qq\in S$ and $\bar\omega\in\bar\Xman$
\begin{gather*}
 \mu (\sublev{\qq}) = \lim_{d\to\infty} \frac{1}{d} \, \# \{r<d \mid \Xfunc\circ\Tmap^r(\bar\omega) \leq \qq \}.
\end{gather*}
In particular, if $\bar\omega\in\bar\Xman$ is such that $\qq=\Xfunc(\bar\omega)\in S$, then
\begin{align*}
 F(\Xfunc(\bar\omega)) &= F(\qq) = \mu (\sublev{\qq}) = \mu (\sublev{\Xfunc(\bar\omega)}) = \\
 &= \lim_{d\to\infty} \frac{1}{d} \, \# \{r<d \mid \Xfunc\circ\Tmap^r(\bar\omega) \leq \Xfunc(\bar\omega) \}
 = \lim_{d\to\infty} \frac{\Id{d}(\bar\omega)}{d}.
\end{align*}
Thus, $\{\Id{d}\}$ converges to $F\circ\Xfunc$ on the set $\Xfunc^{-1}(S) \cap \bar\Xman$.
We will prove that in fact this sequence converges to $F\circ\Xfunc$ on $\bar\Xman$.

Let $\bar\omega\in\bar\Xman$ be such that $\qq=\Xfunc(\bar\omega) \in \RRR\setminus S$.
Then $F$ is continuous at $\qq$.

Choose two sequences
\[
 \{\bb_i\}_{i\in\NNN} \ \subset \ S\cap(-\infty, \qq),
\qquad
 \{\cc_i\}_{i\in\NNN} \ \subset \ S\cap(\qq,+\infty)
\]
converging to $\qq$.
Then by construction of $\bar\Xman$ for each $i\in\NNN$ we have that
\begin{gather*}
 F(\bb_i) = \mu (\sublev{\bb_i})  = \lim_{d\to\infty} \frac{1}{d} \, \# \{r < d \mid \Xfunc\circ\Tmap^r(\bar\omega) \leq \bb_i \}, \\
 F(\cc_i) = \mu (\sublev{\cc_i})  = \lim_{d\to\infty} \frac{1}{d} \, \# \{r < d \mid \Xfunc\circ\Tmap^r(\bar\omega) \leq \cc_i \}.
\end{gather*}
Since $\bb_i < \qq < \cc_i$, we see that
\[
 \# \{r < d \mid \Xfunc\circ\Tmap^r(\bar\omega) \leq \bb_i \}
  \ \ \leq \ \ \Id{d}(\bar\omega)  \ \ \leq \ \ \# \{r < d \mid \Xfunc\circ\Tmap^r(\bar\omega) \leq \cc_i \}.
\]
Hence
\[
F(\qq) = \lim_{i\to\infty} F(\bb_i) \leq
\varliminf_{i\to\infty}  \Id{d}(\bar\omega) \leq
\varlimsup_{i\to\infty}  \Id{d}(\bar\omega)
\leq \lim_{i\to\infty} F(\cc_i) = F(\qq).
\]
Thus $\lim\limits_{d\to\infty} \Id{d}(\bar\omega)$ exists and coincides with $F(\qq)=F(\Xfunc(\bar\omega))$, which proves the lemma.
\end{proof}

\begin{corollary}\label{cor:sigmaX_PdXT}
Let $(\Xman, \FOmega, \mu)$ be a probability space and $\Tmap:\Xman\to\Xman$ and \[\Xmap=(\Xfunc_1,\ldots,\Xfunc_n):\Xman\to\RRR^n\] be measurable maps.
If $\Tmap$ is ergodic and $\mu$-preserving, then
\[\sigma(\Xmap) \ \subcirc \ \SAXmapT.\]
\end{corollary}
\begin{proof}
Suppose $n=1$, so $\Theta=\Xfunc:\Xman\to\RRR$ is a function.
Then by Lemma~\ref{lm:sigmaX_sigma_FX} $\sigma(\Xfunc) \equcirc \sigma(F\circ\Xfunc)$.
Notice that $\Id{d}$ is $\SAXfuncT$-$\Borel([0,1])$-measurable for each $d$ and by Lemma~\ref{lm:FX_lim_Id} the sequence $\{\Id{d}\}$ converges a.e.\! to $F\circ\Xfunc$.
Hence $F\circ\Xfunc$ is $\FOmega$-$\Borel([0,1])$-measurable as well.
This means that $\sigma(\Xfunc) \equcirc \sigma(F\circ\Xfunc) \subcirc \SAXfuncT$.

If $n\geq2$, then for each $i=1,\ldots,n$ we have the inclusion of $\sigma$-algebras:
\[\sigma(\Xfunc_i) \ \subcirc \ \sgmalg{\Xfunc_i}{\Tmap}.\]
Since $\SAXmapT$ is generated by $\SAXfunciT$ for all $i=1,\ldots,n$, we see that
$\sigma(\Xmap) \ \subcirc \ \SAXmapT$.
\end{proof}

\begin{corollary}\label{cor:P_XT_T__PX_T}
Let $(\Xman, \FOmega, \mu)$ be a probability space, $\Tmap:\Xman\to\Xman$ and $\Xmap=(\Xfunc_1,\ldots,\Xfunc_n):\Xman\to\RRR^n$ be measurable maps.
If $\Tmap$ is ergodic and $\mu$-preserving, then
\begin{equation}\label{equ:sigma_XTk__OP}
\sigma \bigl(\{\Xmap\circ\Tmap^k\}_{k\geq0}\bigr) \equcirc\SAXmapT.
\end{equation}
\end{corollary}
\begin{proof}
Since $\SAXmapT\subset \sigma \bigl(\{\Xmap\circ\Tmap^k\}_{k\geq0}\bigr)$, see Lemma~\ref{lm:SigmaXT_sigmaXT}, it suffices to show that
\begin{equation}\label{equ:OPXTk_OPXT}
\sigma \bigl(\{\Xmap\circ\Tmap^k\}_{k\geq0}\bigr) \subcirc\SAXmapT.
\end{equation}
As in the proof of Corollary~\ref{cor:sigmaX_PdXT}, we can restrict to the case $n=1$ with $\Theta=\Xfunc:\Xman\to\RRR$, because
$\SAXmapT$ is generated by $\SAXfunciT, i=1,\ldots,n$.

We will show that
\begin{equation}\label{equ:OPXTk_OPXTxi}
 \sgmalg{\Xfunc\circ\Tmap^k}{\Tmap}\, \subset\, \SAXfuncT, \qquad k\geq1.
\end{equation}
Then by Corollary~\ref{cor:sigmaX_PdXT} we get the inclusions
\[\sigma(\Xfunc\circ\Tmap^k) \ \ \subcirc \ \ \sgmalg{\Xfunc\circ\Tmap^k}{\Tmap} \ \ \stackrel{\eqref{equ:OPXTk_OPXTxi}}{\subset} \ \ \SAXfuncT,\]
which imply~\eqref{equ:OPXTk_OPXT} with $\Theta=\Xfunc:\Xman\to\RRR$.

For the proof of~\eqref{equ:OPXTk_OPXTxi} it is sufficient to show that partition $\PdXT{d+1}$ is finer than $\Pd{\Xfunc\circ\Tmap}{\Tmap}{d}$ for all $d\geq1$:
\[
 \Pd{\Xfunc\circ\Tmap}{\Tmap}{d} \ \prec \ \, \PdXT{d+1}, \qquad d\geq1.
\]

Let $\pi=(i_0,\ldots,i_{d})$ be a permutation of the set $\{0,\ldots,d\}$ and $\Ppi{\Xfunc\circ\Tmap}{\Tmap}{\pi}$ be the corresponding element of partition $\Pd{\Xfunc\circ\Tmap}{\Tmap}{d}$, so $\PpiXT{\pi}$ consists of all $\omega\in\Xman$ such that
\begin{equation*}
\Xfunc\circ\Tmap\circ\Tmap^{i_0}(\omega) \ \geq \ \Xfunc\circ\Tmap\circ\Tmap^{i_1}(\omega) \ \geq \ \cdots \ \geq \ \Xfunc\circ\Tmap\circ\Tmap^{i_{d}}(\omega),
\end{equation*}
and if $\Xfunc\circ\Tmap\circ\Tmap^{i_{\tau}}(\omega)=\Xfunc\circ\Tmap\circ\Tmap^{i_{\tau+1}}(\omega)$, then $i_{\tau}>i_{\tau+1}$.

In other words, $\omega\in\Ppi{\Xfunc\circ\Tmap}{\Tmap}{\pi}$ if and only if
\begin{equation}\label{equ:cond_on_P_XT_T_pi}
\Xfunc\circ\Tmap^{i_0+1}(\omega) \ \geq \ \Xfunc\circ\Tmap^{i_1+1}(\omega) \ \geq \ \cdots \ \geq \ \Xfunc\circ\Tmap^{i_{d}+1}(\omega),
\end{equation}
and whenever $\Xfunc\circ\Tmap^{i_{\tau}+1}(\omega)=\Xfunc\circ\Tmap^{i_{\tau+1}+1}(\omega)$, then $i_{\tau}>i_{\tau+1}$ for $\tau\in\{0,\ldots,d-1\}$.

Consider the following permutations of the set $\{0,\ldots,d+1\}$:
\begin{gather*}
 \alpha_0 = (0,\, i_0+1,\, i_1+1,\, \ldots,\, i_{d}+1), \\
 \alpha_1 = (i_0+1,\, 0,\, i_1+1,\, \ldots,\, i_{d}+1), \\
 \cdots  \cdots \cdots\cdots \\
 \alpha_{d+1} = (i_0+1,\, i_1+1,\, \ldots,\, i_{d}+1,\, 0),
\end{gather*}

We claim that
\begin{equation}\label{equ:PXT_T_pi_decomposition}
 \Ppi{\Xfunc\circ\Tmap}{\Tmap}{\pi} = \Ppi{\Xfunc}{\Tmap}{\alpha_0} \cup \Ppi{\Xfunc}{\Tmap}{\alpha_1} \cup \cdots \cup \Ppi{\Xfunc}{\Tmap}{\alpha_{d+1}}.
\end{equation}
which will prove that the partition $\PdXT{d+1}$ is finer than $\Pd{\Xfunc\circ\Tmap}{\Tmap}{d}$.

Evidently, for all $\omega \in \mathop{\bigcup}\limits_{j=0}^{d+1}\Ppi{\Xfunc}{\Tmap}{\alpha_j}$ condition~\eqref{equ:cond_on_P_XT_T_pi} holds true, that is  $\omega\in\Ppi{\Xfunc\circ\Tmap}{\Tmap}{\pi}$.

Conversely, let $\omega\in\Ppi{\Xfunc\circ\Tmap}{\Tmap}{\pi}$.
If $\Xfunc(\omega) > \Xfunc\circ\Tmap^{i_{0}+1}(\omega)$, then $\omega\in\Ppi{\Xfunc}{\Tmap}{\alpha_0}$.
Otherwise, let $\tau = \max\{ b\in\{0,\ldots,d\}  \mid \Xfunc\circ\Tmap^{i_{b}+1}(\omega)\geq \Xfunc(\omega) \}.$
Then $\omega\in\Ppi{\Xfunc}{\Tmap}{\alpha_{\tau}}$.
This completes the proof of~\eqref{equ:OPXTk_OPXT}.
\end{proof}

\subsection{Proof of Theorem~\ref{th:sigmaX_F_hKS_limPE}}\label{sect:proof_th:sigmaX_F_hKS_limPE}
Let $(\Xman, \FOmega, \mu)$ be a probability space, $\Tmap:\Xman\to\Xman$ be a measurable $\mu$-invariant transformation, and $\Xmap:\Xman\to\RRR^n$ be a measurable map such that $\sigma \bigl(\{\Xmap\circ\Tmap^k\}_{k\geq0}\bigr) \equcirc \FOmega$.
We have to prove that
\begin{equation}\label{equ:Hmu_limPE}
 \hKS_{\mu}(\Tmap) =  \lim_{d\to\infty} \lim_{k\to\infty}\,\frac{1}{k}\, \Hmu\left(\tk{k}{\Ppi{\Xmap}{\Tmap}{d}}\right)
\end{equation}
if either
\begin{enumerate}
 \item[\rm(a)] $\Tmap$ is ergodic, or
 \item[\rm(b)] $\Tmap$ is not ergodic, however $\Xman$ can be embedded into some compact metrizable space so that $\FOmega=\Borel(\Xman)$.
\end{enumerate}
In the case (a) it follows from Corollary~\ref{cor:P_XT_T__PX_T} and the assumptions above that
\[\SAXmapT \equcirc \FOmega,\]
which by Lemma~\ref{lm:hKS_limd} implies~\eqref{equ:Hmu_limPE}.

In the case (b) the equality~\eqref{equ:Hmu_limPE} follows from the Ergodic decomposition theorem by the arguments of the proof of~\cite[Theorem~2.1]{Keller:DCDS:2012}.
\qed

\section{Residuality and prevalence}\label{s4}

\subsection{The set of non-injectivity}
Let $\Xmap:\Xman\to\Yman$ be a continuous map between topological spaces $\Xman$ and $\Yman$.
Suppose also that $\mu$ is a measure on the $\sigma$-algebra $\Borel(\Xman)$ of Borel sets of $\Xman$.
In this section we give sufficient conditions on $\Xmap$ for the equvalence $\sigma(\Xmap) \equcirc \Borel(\Xman)$ and also prove Theorem~\ref{th:setX_residual}.

The subset
\begin{equation}\label{equ:D_defn}
\noninj{\Xmap} = \{ \, \omega \in \Xman \mid \Xmap^{-1}\Xmap(\omega) \not=\{\omega\} \, \} 
\end{equation}
of $\Xman$ is called \emph{the set of non-injectivity} of $\Xmap$.
It plays a principal role in the further considerations.

We will now present a more useful description of $\noninj{\Xmap}$.
Given a set $\Xman$, a number $s\geq 1$, and map $\Xmap:\Xman\to \Yman$ into some set $\Yman$ put
\[
 \Xman^s = \underbrace{\Xman \times \cdots \times \Xman}_{s}\,,
 \qquad
 \Delta \Xman^s = \{(\omega,\ldots,\omega) \in \Xman^s \mid \omega \in \Xman \},
\]
\[
 \Xmap^s=\underbrace{\Xmap\times\cdots\times \Xmap}_{s}:\Xman^s\to \Yman^s,
 \qquad
 \Xmap^s(\omega_1,\ldots,\omega_s) = (\Xmap(\omega_1),\ldots,\Xmap(\omega_s)).
\]
%
%
In particular, for $s=2$ consider the following subset of $\Omega^2$:
\begin{equation}\label{equ:symm_noninj}
 \C  \ = \ (\Xmap^2)^{-1}(\Delta\Yman^2) \, \setminus \, \Delta\Xman^2 \ = \ \{ (\omega,\omega') \mid \omega\not=\omega', \Xmap(\omega)= \Xmap(\omega') \}.
\end{equation}
Let also $p:\Omega^2\to\Omega$ be the projection to the first coordinate.
Then it is evident that
\begin{equation}\label{equ:NonInj_defn_2}
 \noninj{\Xmap} = p(\C).
\end{equation}
The following lemma describes some properties of the set of non-injectivity.
\begin{lemma}\label{lm:noninj_prop}
{\rm 1)}
$\Xmap^{-1}\Xmap(\noninj{\Xmap}) = \noninj{\Xmap}$.

{\rm2)} For any subset $F\subset \Xman\setminus \noninj{\Xmap}$ the restriction $\Xmap|_{F}:F\to \Yman$ is injective.

{\rm 3)}
Suppose $\Xman$ and $\Yman$ are Hausdorff, $\Xman$ is also \textcolor{blue}{second countable} and locally compact (e.g. a manifold).
Then $\noninj{\Xmap}$ is an $F_{\sigma}$ subset of $\Xman$, and in particular $\noninj{\Xmap}\in\Borel(\Omega)$.
\end{lemma}
\begin{proof}
Statements 1) and 2) are evident.
Let us prove 3).
It is easy to see that a topological space $\Yman$ is Hausdorff iff the diagonal $\Delta\Yman^2$ is closed in $\Yman^2$.
This implies that $(\Xmap^2)^{-1}(\Delta\Yman^2)$ is closed in $\Xman^2$, hence
$\C$ defined by~\eqref{equ:symm_noninj} is \textcolor{blue}{second countable} and locally compact as well.
Therefore $\C=\bigcup\limits_{i=1}^{\infty} \C_i$ where each $\C_i$ is compact.
Hence
\[
\noninj{\Xmap} = p(\C) = p\left(\bigcup\limits_{i=1}^{\infty} \C_i \right) = \bigcup\limits_{i=1}^{\infty} p(\C_i).
\]
But each $p(\C_i)$ is compact and so closed in $\Xman$.
Hence $\noninj{\Xmap}$ is an $F_{\sigma}$-set.
\end{proof}

Recall that a \emph{Polish} space is a \textcolor{blue}{second countable} completely metrizable topological space.
\begin{theorem}\label{th:non_inj}
Let $\Xman$ and $\Yman$ be Polish spaces, $\mu$ be a measure on $\Borel(\Xman)$, $\Xmap:\Xman\to\Yman$ be a continuous map, and $\noninj{\Xmap}$ be the set of its non-injectivity.
Suppose $\noninj{\Xmap} \in \Borel(\Xman)$ and $\mu(\noninj{\Xmap})=0$.
Then $\sigma(\Xmap) \equcirc \Borel(\Xman)$.
\end{theorem}
\begin{proof}
Since $\sigma(\Xmap) \subset \Borel(\Xman)$, it remains to consider the inverse inclusion.
It suffices to show that for any open set $G\in\Borel(\Xman)$ there exists some $\widetilde{G}\in\sigma(\Xmap)$ such that $\mu(G\bigtriangleup \widetilde{G}) = 0$.
Given $G$, put
\[
 \widetilde{G} = G \setminus \noninj{\Xmap}.
\]
Then by 2) of Lemma~\ref{lm:noninj_prop} the restriction $\Xmap|_{\widetilde{G}}:\widetilde{G}\to \Xmap(\widetilde{G})$ is one-to-one.
So $\Xmap(\widetilde{G})$ is a one-to-one image of the Polish space $\widetilde{G}$ under the continuous map $\Xmap|_{\widetilde{G}}:\widetilde{G}\to\Yman$.
This implies, \cite[Theorem~15.1]{Kechris:1995}, that $\Xmap(\widetilde{G}) \in \Borel(\Yman)$, whence
\[
 \widetilde{G} = \Xmap^{-1}(\Xmap(\widetilde{G})) \ \in \ \sigma(\Xmap).
\]
Therefore $\mu(G\bigtriangleup \widetilde{G}) \,=\, \mu(G\cap\noninj{\Xmap}) \,\leq\, \mu(\noninj{\Xmap}) \,=\, 0.$
\end{proof}


\subsection{Multijets transversality theorem}
The proof of Theorem~\ref{th:setX_residual} is based on the so-called multijets transversality theorem, see~\cite[Chapter 2, Theorem~4.13]{GolubitskyGuillemin:GTM:1973}.
We will formulate it below preserving the notation from~\cite{GolubitskyGuillemin:GTM:1973}.

Let $\Xman$ and $\Yman$ be smooth manifolds, $\dim \Xman=m$, $\dim \Yman=n$, $J^k(\Xman,\Yman)$ be the manifold of $k$-jets of smooth maps $\Xmap:\Xman\to \Yman$,
\[
 \alpha:J^k(\Xman,\Yman) \to \Xman
\]
be the natural projection to the source,
\[
 \alpha^s=\alpha\times\cdots\times\alpha:J^k(\Xman,\Yman)^s \to \Xman^s,
\]
\[
 \Xman^{(s)} = \{(\omega_1,\ldots,\omega_n) \in \Xman^s \mid \omega_i\not=\omega_j \ \text{for} \ i\not=j \},
\]
and
\[
 J^k_s(\Xman,\Yman) = (\alpha^s)^{-1} \Xman^{(s)}.
\]
Then $J^k_s(\Xman,\Yman)$ is an open submanifold of $J^k(\Xman,\Yman)^{s}$ and we have the map
\[
 j^k_s \Xmap : \Xman^{(s)} \to J^k_s(\Xman,\Yman),
 \qquad
 j^k_s \Xmap(\omega_1,\ldots,\omega_s) =
 \bigl(j^k\Xmap(\omega_1),\ldots,j^k\Xmap(\omega_s)\bigr).
\]

The following result is called \emph{multijets transversality theorem}.

\begin{theorem}\label{tm:multijets_transv}{\rm \cite[Chapter 2, Theorem~4.13]{GolubitskyGuillemin:GTM:1973}}.
Let $\Wman$ be a submanifold in $J^k_s(\Xman,\Yman)$.
Endow $C^{\infty}(\Xman,\Yman)$ with the strong topology $C^{\infty}_S$.
Then the set
\[
 \cmgset_{\Wman} = \{ \Xmap\in C^{\infty}(\Xman,\Yman) \mid j^k_s \Xmap \ \text{is transversal to} \ \Wman \}
\]
is residual in $C^{\infty}(\Xman,\Yman)$.
If $\Wman$ is compact, then $\cmgset_{\Wman}$ is also open.
\qed
\end{theorem}
We will apply this theorem to the case $k=0$ and $s=2$.
First we will show that description~\eqref{equ:NonInj_defn_2} of $\noninj{\Xmap}$ is related to multijets transversality theorem.
Recall that $J^0(\Xman,\Yman) = \Xman\times \Yman$.
Then
\[
 J^0_2(\Xman,\Yman) = \{(\omega,\yy,\omega',\yy')\mid \omega\not=\omega' \} \subset  J^0(\Xman,\Yman)^2 = (\Xman\times \Yman)^2.
\]
Let $\beta:J^0_2(\Xman,\Yman) \to \Yman^2$ be the projection to the destination given by
\[
 \beta(\omega,\yy,\omega',\yy') = (\yy,\yy'),
\]
and
\begin{equation}\label{equ:W_dbl_points}
 \Wman = \beta^{-1}(\Delta \Yman^2) = \{(\omega,\yy,\omega',\yy)\mid \omega\not=\omega' \}.
\end{equation}
Evidently, $\beta$ is a submersion.
Therefore it is transversal to $\Delta \Yman^2$, and so $\Wman$ is a submanifold in $J^0_2(\Xman,\Yman)$ of codimension
\[\mathrm{codim}\,\Delta \Yman^2 = \dim \Yman^2-\dim \Delta \Yman^2 = \dim \Yman = n.\]
Also notice that $\Wman$ is non-compact.
Consider the map
\[
 j^0_2\Xmap: \Xman^{(2)}\to J^0_2(\Xman,\Yman), \,
 \qquad
 j^0_2\Xmap(\omega,\omega') = (\omega,\Xmap(\omega), \omega',\Xmap(\omega')).
\]
Then
\[
 \C = (j^0_2\Xmap)^{-1}(\Wman) = \{(\omega,\omega')\in \Xman^2 \mid \omega\not=\omega', \Xmap(\omega)=\Xmap(\omega') \},
\qquad
 \noninj{\Xmap} = p(\C),
\]
as in~\eqref{equ:symm_noninj} and~\eqref{equ:NonInj_defn_2}.

\begin{corollary}\label{cor:noninjset}
Let $\Wman$ be the submanifold of $J^0_2(\Xman,\Yman)$ given by~\eqref{equ:W_dbl_points} and
\begin{equation}\label{equ:X_to_Y_transv_W}
\cmgset_{\Wman} = \{ \Xmap\in C^{\infty}(\Xman,\Yman) \mid j^0_2 \Xmap \ \text{is transversal to} \ \Wman \}.
\end{equation}
Then by Theorem~{\rm\ref{tm:multijets_transv}} $\cmgset_{\Wman}$ is residual in $C^{\infty}(\Xman,\Yman)$ with respect to the strong topology $C^{\infty}_S$.
If $m<n$, then for each $\Xmap\in \cmgset_{\Wman}$
the set $\noninj{\Xmap}$ has measure zero in the sense of Definition~{\rm\ref{defn:zero_measure}}.
\end{corollary}
\begin{proof}
%
%
Notice that $\Wman$ has codimension $n$ in $J^0_2(\Xman,\Yman)$.
Therefore the submanifold $\C= (j^0_2\Xmap)^{-1}(\Wman)$ has the same codimension $n$ in $\Xman^{(2)}$.
Since $m<n$, we obtain that
\[
 \dim \C = \dim \Xman^{(2)} - n = 2m-n < m =\dim \Xman.
\]
Consider the restriction $p|_{\C}:\C \to \Xman$.
As $\dim \C < \dim \Xman$, each point of $\C$ is critical for $p|_{\C}$, whence by Sard's theorem, \cite{GolubitskyGuillemin:GTM:1973}, the image of the set of critical points of $p|_{\C}$, i.e. the set $p(\C) = \noninj{\Xmap}$, has measure zero in the sense of Definition~\ref{defn:zero_measure}.
\end{proof}
%
\begin{corollary}\label{cor:preval}
Suppose in Corollary~\ref{cor:noninjset} $\Yman=\RRR^n$ for some $n$, so $C^{\infty}(\Xman,\RRR^n)$ is a linear space.
Then the set $\cmgset_{\Wman}$ has a probe.
In particular, by Lemmas~\ref{lm:prevalency_prop} and~\ref{lm:topologies} it is prevalent with respect to any of weak topologies $C^{r}_{W}$ on $C^{\infty}(\Xman,\RRR^n)$.
\end{corollary}
\begin{proof}
First we introduce some notation and prove Lemma~\ref{lm:constr_of_G} below.
Let $\Mnk$ be the space of $(n\times k)$-matrices ($n$ rows and $k$ columns) which can be identified with $\RRR^{nk}$, and $D_r(n,k)$ be the subset of $\Mnk$ consisting of matrices of rank $r$.
Then $D_r(n,k)$ is a smooth submanifold of codimension $(n-r)(k-r)$, see e.g.~\cite[Lemma~1.19]{Milnor:DiffTop}.
Define the map
\[
  \Phi: \Mnk \times \Mnk \to M(2n,k),
\]
which associate to each pair $A,B\in \Mnk$ the matrix $\Phi(A,B)$ obtained by appending all rows of $B$ to $A$, see Figure~\ref{fig:matr_PhiAB}.
\begin{figure}[h]
\includegraphics[height=1.5cm]{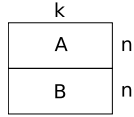}
\protect\label{fig:matr_PhiAB}\caption{Matrix $\Phi(A,B)$}
\end{figure}
Evidently, $\Phi$ is a smooth diffeomorphism.

Now we can construct the probe for $\cmgset_{\Wman}$.
Let $G:\Omega\to\Mnk$ be a map satisfying statement of Lemma~\ref{lm:constr_of_G}.
For each $v\in\RRR^k$ define the following smooth map
\[
 L_v:\Omega\to\RRR^n,
 \qquad
 L_v(\omega) = G(\omega) v,
\]
and let
\[
 P \, = \,  \{ L_v \mid v\in\RRR^k \}  \ \subset \ C^{\infty}(\Omega,\RRR^n).
\]
Then $P$ is a linear subspace of $C^{\infty}(\Omega,\RRR^n)$ of dimension $\leq k$.
We claim that $P$ is a probe for $\cmgset_{\Wman}$.

Indeed, let $\Xmap\in C^{\infty}(\Omega,\RRR^n)$ be any map.
For each $v\in\RRR^k$ denote
\[
 \Xmap_v = \Xmap + L_v,
\]
so the translation of $P$ by $\Xmap$ is the following affine subspace of $C^{\infty}(\Omega,\RRR^n)$:
\[
 \Xmap + P = \{ \Xmap_v \mid v\in\RRR^k\}.
\]
We should prove that the following set:
\[
Q = \{ v \in \RRR^k \mid j^0_2\Xmap_v \ \text{is not transversal to} \ \Wman \}
\]
has Lebesgue measure zero in $\RRR^k$.
Define the map
\[\Psi:\Omega^{(2)} \times\RRR^k \to J^2_0(\Omega,\RRR^n) \subset (\Omega\times\RRR^n)^2\]
by
\[
\Psi(\omega,\omega',v) := j^{0}_{2}(\Xmap + L_v)(\omega,\omega') =
\bigl(\omega,\, \Xmap(\omega) + L_v(\omega),\, \omega',\, \Xmap(\omega')+ L_v(\omega') \bigr).
\]
Then the Jacobi matrix of $\Psi$ at point $(\omega,\omega',v)$ has the form shown in Figure~\ref{fig:JPsi}, and so its rank is maximal and equals $2(m+n)$ due to the choice of $G$.
\begin{figure}[h]
\includegraphics[height=3cm]{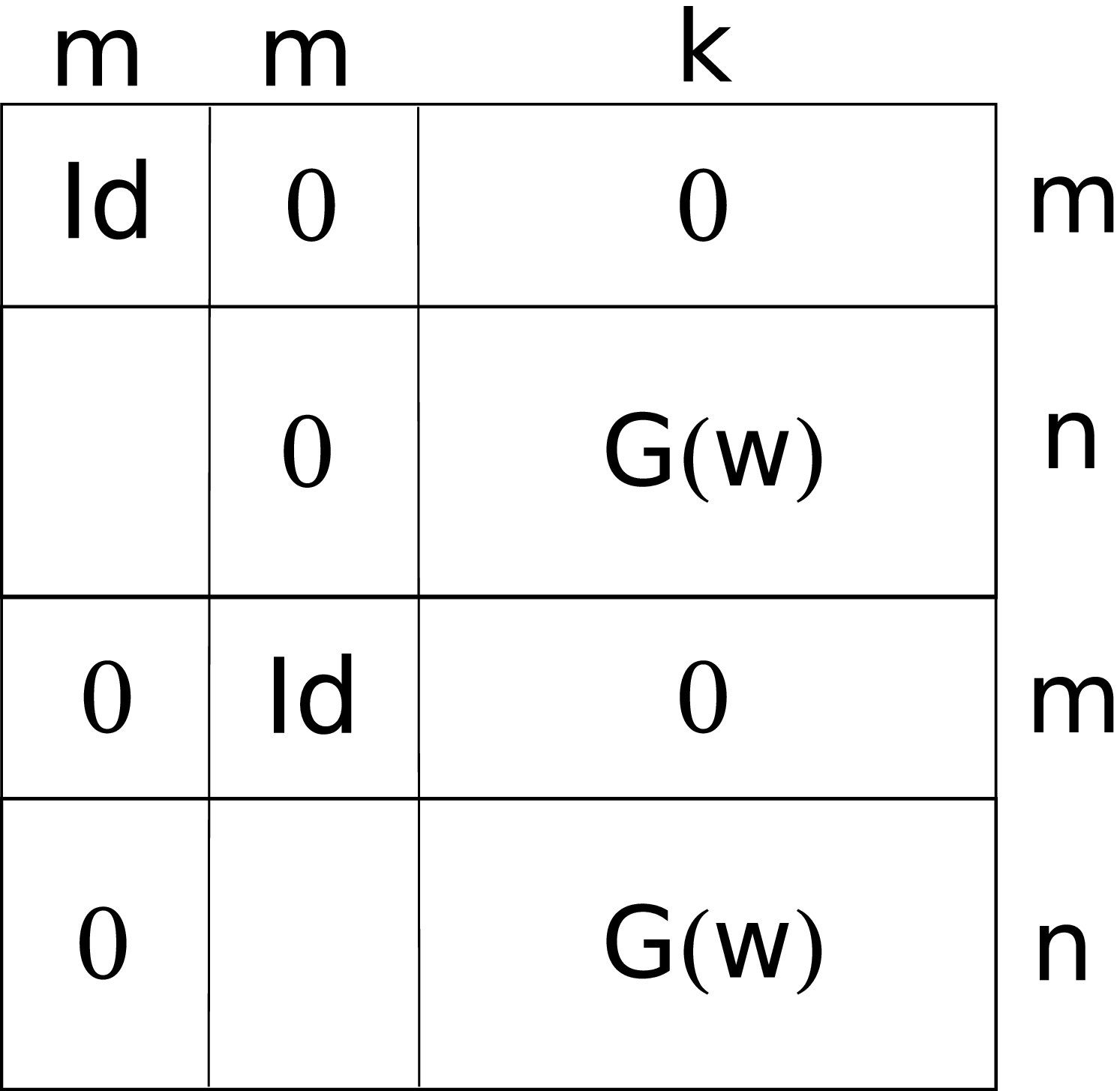}
\protect\label{fig:JPsi}\caption{Jacobi matrix of $\Psi$ at point $(\omega,\omega',v)$}
\end{figure}
Hence $\Psi$ is a submersion.
Therefore it is transversal to $\Wman$, and so $\tC = \Psi^{-1}(\Wman)$ is a submanifold in $\Omega^{(2)} \times\RRR^k$.
Let $\pi:\tC\to\RRR^k$ be the restriction to $\tC$ of the natural projection $\Omega^{(2)} \times\RRR^k \to \RRR^k$.
Then it is easy to see that $Q$ coincides with the set of critical values of $\pi$.
Since $\pi$ is smooth, we get from Sard's theorem that $Q$ has Lebesgue measure zero, see e.g.~\cite[Chapter~2, \S3]{GuilleminPolak:1974}.
Corollary~\ref{cor:preval} is completed.
\end{proof}

\begin{lemma}\label{lm:constr_of_G}
If $k \geq 2(m+n)$, then there exists a smooth map $G:\Omega\to\Mnk$ such that for any pair of distinct points $\omega\not=\omega'\in \Omega$ the matrix
\[
 \Phi(G(\omega),G(\omega'))
\]
has rank $2n$.
\end{lemma}
\begin{proof}
The proof is also based on the multijets transversality theorem.
Consider the following spaces:
\begin{align*}
J^0(\Xman,\Mnk) &= \Xman\times \Mnk,  \\
J^0_2(\Xman,\Mnk) &= \{(\omega,A,\omega',A')\mid \omega\not=\omega' \} \subset  J^0(\Xman,\Mnk)^2 = (\Xman\times \Mnk)^2.
\end{align*}
Let
\[
\gamma=\Phi\circ\beta:J^0_2(\Xman,\Mnk) \xrightarrow{~\beta~} \Mnk^2 \xrightarrow{~\Phi~} M(2n,k)
\]
be the projection to the destination $\beta$ composed with the diffeomorphism $\Phi$:
\[
 \gamma(\omega,\MA,\omega',\MA') = \Phi(\MA,\MA'),
\]
and
\begin{equation*} 
 \widetilde{D}_r(2n,k) = \gamma^{-1}(D_r(2n,k)) = \{(\omega,\MA,\omega',\MA')\mid \omega\not=\omega', \ \mathrm{rank}\,\Phi(\MA,\MA') = r\}.
\end{equation*}
for $r<2n$.
Since $\gamma$ is a submersion, it is transversal to $D_r(2n,k)$, and so $\widetilde{D}_r(2n,k)$ is a submanifold in $J^0_2(\Xman,\Mnk)$ of codimension
\[
\mathrm{codim}\,\widetilde{D}_r(2n,k) = \mathrm{codim}\,D_r(2n,k) = (2n-r)(k-r).
\]
Then by multijets transversality theorem the set
\[
 T_{r} = \{ G\in C^{\infty}(\Xman,\RRR^n) \mid j^0_2 G \ \text{is transversal to} \ \widetilde{D}_r(2n,k) \}
\]
is residual.
Hence, so is the intersection
\[
 T = \bigcap_{r=0}^{2n-1} T_{r},
\]
and, in particular, $T$ is non-empty.
We claim that any $G\in T$ satisfies the statement of the lemma.

First notice that the assumption $k\geq 2(m+n)$ is equivalent to the inequality: $k - 2n + 1 > 2m$.
Then for $0\leq r \leq 2n-1$ we have that
\[
\mathrm{codim}\, \widetilde{D}_{r}(2n,k) \ \geq \ \mathrm{codim}\, \widetilde{D}_{2n-1}(2n,k) \ = \ k - 2n + 1 \ >\  2m\  =\  \dim\Omega^{(2)},
\]
and so transversality of $j^0_2 G$ to $\widetilde{D}_r(2n,k)$ means that $j^0_2 G(\Omega^{(2)}) \cap \widetilde{D}_r(2n,k) = \varnothing$.
Thus if a map $G:\Omega\to\RRR^n$ belongs to $T$, then
\[ j^0_2 G(\Omega^{(2)}) \cap \widetilde{D}_r(2n,k) = \varnothing, \qquad r=0,1,\ldots,2n-1. \]
This means that $\mathrm{rank}\,\Phi(G(\omega),G(\omega'))=2n$ for any $\omega\not=\omega'\in\Omega$.
\end{proof}

\subsection*{Proof of Theorem~\ref{th:setX_residual}}\label{sect:proof_th:setX_residual}
Let $\Xman$ be a smooth manifold of dimension $m$, $\mu$ be a Lebesgue absolute continuous measure on $\Borel(\Xman)$, $\Tmap:\Xman\to\Xman$ be a measurable $\mu$-invariant transformation, and $n>m$.

Let $\cmgset = \cmgset_{\Wman}$ be defined by~\eqref{equ:X_to_Y_transv_W}.
Then by Corollaries~\ref{cor:noninjset} and~\ref{cor:preval} $\cmgset$ is residual with respect to the strong topology $C^{\infty}_S$ and prevalent with respect to the weak topology $C^{\infty}_W$.

We claim that~\eqref{hKS_is_a_limit} holds for each $\Xmap\in\cmgset$.
Indeed, by 3) of Lemma~\ref{lm:noninj_prop} $\noninj{\Xmap}$ is a Borel subset of $\Xman$.
Also by Corollary~\ref{cor:noninjset} it has measure zero in the sense of Definition~\ref{defn:zero_measure}.
Since $\mu$ is Lebesgue absolute continuous, we see that $\mu(\noninj{\Xmap})=0$, whence by Theorem~\ref{th:non_inj} $\sigma(\Xmap) \equcirc \Borel(\Xman)$.

Furthermore, as $\Xman$ is an $m$-dimensional manifold, it can be embedded in $(2m+1)$-cube being a compact metric space.
Therefore by (b) of Theorem~\ref{th:sigmaX_F_hKS_limPE} we have that
\[
 \hKS_{\mu}(\Tmap) =  \lim_{d\to\infty} \lim_{k\to\infty}\,\frac{1}{k}\, \Hmu\left(\tk{k}{\Pd{\Xmap}{\Tmap}{d}}\right).
\]
This completes Theorem~\ref{th:setX_residual}.
\qed



\providecommand{\bysame}{\leavevmode\hbox to3em{\hrulefill}\thinspace}
\providecommand{\MR}{\relax\ifhmode\unskip\space\fi MR }
\renewcommand{\MR}{\relax\ifhmode\unskip\space\fi MR }
\providecommand{\MRhref}[2]{%
  \href{http://www.ams.org/mathscinet-getitem?mr=#1}{#2}
}
\providecommand{\href}[2]{#2}


\begin{thebibliography}{1}
\bibitem{Amigo2012}
 \newblock J.M.~Amig\'{o},
 \newblock \emph{The equality of Kolmogorov-Sinai entropy and metric permutation entropy generalized},
 \newblock Physica D, \textbf{241}, (2012), no. 7, 789--793. \MR{MR2897545}

\bibitem{Amigo:2010}
 \newblock Jos{\'e}~Mar{\'{\i}}a Amig{\'o},
 \newblock \emph{Permutation complexity in dynamical systems},
 \newblock Springer Series in Synergetics, Springer-Verlag, Berlin, 2010, Ordinal patterns, permutation entropy and all that. \MR{2583155 (2011f:37002)}

\bibitem{AmigoKennelKocarev2005}
 \newblock J.M.~Amig\'{o}, M.B.~Kennel, and L.~Kocarev,
 \newblock \emph{The permutation entropy rate equals the metric entropy rate for ergodic information sources and ergodic dynamical systems},
 \newblock Physica D, \textbf{210} (2005), 77--95.

\bibitem{BandtKellerPompe:Nonlin:2002}
 \newblock Christoph Bandt, Gerhard Keller, and Bernd Pompe,
 \newblock \emph{Entropy of interval maps via permutations},
 \newblock Nonlinearity \textbf{15} (2002), no.~5, 1595--1602.  \MR{1925429 (2003h:37048)}

\bibitem{BandtPompe2002}
 \newblock C.~Bandt, B.~Pompe,
 \newblock \emph{Permutation entropy: A natural complexity measure for time series},
 \newblock Phys.~Rev.~Lett. \textbf{88}, (2002), 174102.

\bibitem{GolubitskyGuillemin:GTM:1973}
 \newblock M.~Golubitsky and V.~Guillemin,
 \newblock \emph{Stable mappings and their singularities},
 \newblock  Springer-Verlag, New York, 1973, Graduate Texts in Mathematics, Vol. 14.  \MR{0341518 (49 \#6269)}

\bibitem{GuilleminPolak:1974}
 \newblock V.~Guillemin and A.~Polak
 \newblock \emph{Differential topology},
 \newblock Prentice-Hall, Englewood Cliff, NJ, 1974


\bibitem{Hirsch:DiffTop:1976}
 \newblock M.~Hirsch,
 \newblock \emph{Differential topology},
 \newblock Graduate Texts in Mathematics, No. 33. Springer-Verlag, New York-Heidelberg, 1976. x+221 pp.\MR{MR0448362 (56 \#6669) }


\bibitem{HuntSauerYourke:BAMS:1992}
 \newblock B.~R.~Hunt, T.~Sauer, J.~A.~Yourke,
 \newblock \emph{Prevalence: a translation-invariant ``almost every'' on infinite-dimensional spaces},
 \newblock Bull. Amer. Math. Soc. \textbf{27}, (1992), no.~2, 217-238.

\bibitem{Keller:DCDS:2012}
 \newblock Karsten Keller,
 \newblock \emph{Permutations and the {K}olmogorov-{S}inai entropy},
 \newblock  Discrete Contin. Dyn. Syst. \textbf{32} (2012), no.~3, 891--900. \MR{2851883 (2012k:37020)}

 \bibitem{KellerUnakafovUnakafova2012}
 \newblock K.~Keller, A.~Unakafov, and V.~Unakafova,
 \newblock \emph{On the Relation of KS Entropy and Permutation Entropy},
 \newblock Physica D, \textbf{241} (2012), 1477--1481.

\bibitem{KellerSinn2010}
 \newblock K.~Keller, M.~Sinn,
 \newblock \emph{Kolmogorov-Sinai entropy from the ordinal viewpoint},
 \newblock Physica D, \textbf{239} (2010), 997--1000.

\bibitem{KellerSinn:Nonlin:2009}
 \newblock K.~Keller and M.~Sinn,
 \newblock \emph{A standardized approach to the {K}olmogorov-{S}inai entropy},
 \newblock Nonlinearity \textbf{22} (2009), no.~10, 2417--2422. \MR{2539761 (2010m:37010)}

\bibitem{Kechris:1995}
 \newblock A.~S.~Kechris,
 \newblock \emph{Classical descriptive set theory.}
 \newblock Graduate Texts in Mathematics, 156. Springer-Verlag, New York, 1995. 402 p.

\bibitem{Milnor:DiffTop}
 \newblock J.~Milnor,
 \newblock \emph{Differential topology},
 \newblock Mimeographed notes. Princeton University, New Jersey 1958.

\bibitem{Takens81}
 \newblock F.~Takens,
 \newblock \emph{Detecting strange attractors in turbulence},
 \newblock in: Dynamical Systems and Turbulence (eds. D. A. Rand, L. S. Young), Lecture Notes in Mathematics 898, Springer-Verlag, Berlin-New York 1981, 366--381.

\bibitem{Sauer91}
 \newblock T.~Sauer, J.~Yorke, and M.~Casdagli,
 \newblock \emph{Embeddology},
 \newblock J. Stat. Phys., \textbf{65} (1991), 579--616.

\bibitem{Walters:GTM:1982}
 \newblock Peter Walters,
 \newblock \emph{An introduction to ergodic theory},
 \newblock Graduate Texts in Mathematics, vol.~79, Springer-Verlag, New York, 1982. \MR{648108 (84e:28017)}

\end{thebibliography}
\end{document}